\documentclass[reqno]{amsart}
\usepackage{amsmath, amssymb, amsfonts}
\usepackage{pstricks}

\newcommand{\Z}{{\mathbb Z}} 
\newcommand{\Q}{{\mathbb Q}}

\newcommand{\R}{{\mathbb R}}
\newcommand{\F}{{\mathbb F}}

\newcommand{\Char}{\operatorname{char}}

\newfont{\cyr}{wncyb10}
\newcommand{\sha}{\mbox{\cyr Sh}}
\newcommand{\GCD}{gcd }

\newtheorem{lemma}{Lemma}
\newtheorem{prop}{Proposition}
\newtheorem{thm}{Theorem}
\newtheorem{thm*}{Theorem}
\newtheorem{cor}{Corollary}

\theoremstyle{definition}
\newtheorem{example}{Example}
\newtheorem{exercise}{Exercise}

\theoremstyle{remark}
\newtheorem*{remark}{Remark}

\begin{document}

\title {Counterexamples to the Hasse principle}
\author{W. Aitken, F. Lemmermeyer}



\maketitle

\section{Introduction}\label{s0}

In this article we develop counterexamples to the Hasse principle
using only techniques from undergraduate
number theory and algebra. By keeping the technical prerequisites to a minimum, we hope
to provide a path for nonspecialists to this interesting area of number theory.
The counterexamples considered here extend the classical counterexample of Lind and Reichardt.
As discussed in an appendix, this type of counterexample is
important in the theory of elliptic
curves: today they are interpreted as nontrivial elements
in the Tate--Shafarevich group.


\section{Background}\label{s1}

The problem of determining if the Diophantine equation
\begin{equation}\label{EL}
    aX^2 + bY^2 + cZ^2 = 0
\end{equation}
has nontrivial solutions with values in~$\Z$
has played a prominent role in the 
history of number theory.
We assume that $a, b,$ and $c$ are nonzero
integers and,
using a simple argument, we
reduce to the case where the product
$a b c$ is square-free.
Lagrange (1768) solved the problem by giving a descent procedure
which determines in a finite number of steps whether or not (\ref{EL}) has
a nontrivial $\Z$-solution, but Legendre (1788) gave the definitive solution.
Legendre proved that the following conditions,  known
by Euler to be necessary, are sufficient
for the existence of a nontrivial~$\Z$-solution:
\emph{(i) $a$, $b$, and $c$ do not all have the same sign,
and  (ii) $-ab$ is a square modulo~$|c|$,
$-ca$ is a square modulo~$|b|$,
and $-bc$ is a square modulo~$|a|$.}
Legendre then made interesting use of
this result in the first
attempted proof of quadratic reciprocity.\footnote{Legendre's proof of quadratic reciprocity
had gaps. The first complete proof was given by Gauss
in his \emph{Disquisitiones Arithmeticae} (1801).
Equation~(\ref{EL}) also figures prominently in 
the \emph{Disquisitiones}~\cite[Art.~294--300]{Gauss};
Gauss proves
Legendre's theorem on the solvability of~(\ref{EL}) using his theory of ternary quadratic forms.
He then discusses the gaps in Legendre's proof of quadratic reciprocity.

For a proof of Legendre's
theorem on the solvability of (\ref{EL}) based on Lagrange's descent see
\cite[Ch.~VII \S 3]{Davenport},
\cite[Ch.~17 \S 3]{IR}, or
\cite[Ch.~II \S XIV, Ch.~IV Appendix I]{Weil}
(which gives historical background including Lagrange's
role in the solution to the problem).
For more on Legendre's theorem see
\cite[Ex.~1.8, 2.36]{LRL},
and various books on Diophantine equations.
Lagrange's descent gives an explicit method for finding a solution
if it exists; see Cremona and Rusin \cite{CR} for practical
improvements on the descent method.

Equation~(\ref{EL}) arises more often than
it might at first appear.
If $F \in \Z[X, Y, Z]$ is a homogeneous quadratic
polynomial, then $F(X, Y, Z)=0$
can be transformed to the form~(\ref{EL})
(see \cite[Thm.~$1'$, Ch.~IV]{Ser}).
As a consequence, the
problem of determining if a conic, defined over~$\Q$, has a rational
point reduces in a straightforward manner to the
solvability of~(\ref{EL}).}

There was a large interest in generalizing Legendre's result to
quadratic forms in arbitrarily many variables. Hasse's solution
(1923) was formulated in a very elegant way using the $p$-adic
numbers developed earlier by his teacher Hensel.

To explain Hasse's result we will need to fix some terminology.
 A \emph{homogeneous polynomial} of degree~$d$
is a sum of monomials that are all of total degree $d$; such
polynomials are sometimes called
\emph{forms} of degree~$d$.
Consider the  Diophantine equation
\begin{equation}\label{gen}
F(X_1, \ldots, X_m) = 0
\end{equation}
where $F  \in \Z[X_1, \ldots, X_m]$
is a homogeneous polynomial of positive degree~$d$.
The $m$-tuple $(0, \ldots, 0)$ is a solution, but not an interesting one.
The $m$-tuple $(a_1, \ldots,  a_m)$ is
called \emph{nontrivial} if at least one $a_i$ is nonzero.
We are interested in finding necessary and sufficient
conditions for the existence of nontrivial integer solutions to~(\ref{gen}).
A nontrivial $m$-tuple
$(a_1, \ldots,  a_m)\in \Z^m$
is said to be \emph{primitive} if the greatest common divisor
of $a_1, \ldots, a_m$ is~$1$.
Observe that by homogeneity if (\ref{gen})
has any nontrivial solution (in $\Z^m$, or even~$\Q^m$) it has a
primitive solution.
We extend this terminology in two ways: to systems of homogeneous
polynomial equations, and to solutions modulo~$N$.
For example, a \emph{primitive solution modulo~$N$} is a primitive $m$-tuple of integers
that solves the congruence $F(X_1, \ldots, X_m)\equiv 0$ 
modulo~$N$.\footnote{Any $m$-tuple of integers solving $F(X_1, \ldots, X_m)\equiv 0$ modulo $N$
that is nontrivial with \GCD prime to~$N$
is congruent
to a solution whose \GCD is~$1$.
Thus, if desired, one
can relax the definition of \emph{primitive $m$-tuple modulo~$N$} to
allow any nontrivial $m$-tuple whose \GCD is prime to~$N$. }
For systems and systems modulo~$N$ of homogeneous equations we do not require that the
equations have the same degree.

An easy way to show that~(\ref{gen}) has no
nontrivial $\Z$-solution is to show that it has no nontrivial
$\R$-solutions.
This trick only establishes the nonsolvability of the most blatant offenders:
any interesting Diophantine equation 
will require some number-theoretic tools.
The next-easiest way to show the nonsolvability of~(\ref{gen}) 
is to show that it fails to have a
primitive solution modulo~$N$ for some positive integer~$N$. What is surprising
is that in degree $2$ these two techniques are all that is needed.

\begin{thm*} [Hasse's Theorem: version 1]
If $F \in \Z[X_1, \ldots, X_m]$
is homogeneous of degree~$2$,
then $F(X_1, \ldots, X_m) = 0$ has a nontrivial $\Z$-solution
if and only if
\begin{enumerate}
\item[(i)]  it has a nontrivial $\R$-solution, and
\item[(ii)] it has a primitive solution modulo~$N$
             for all positive integers~$N$.
\end{enumerate}
\end{thm*}

The assertion that (i) and (ii) are necessary and sufficient for the
existence of nontrivial solutions is called the \emph{Hasse principle}
for polynomials of degree~$2$.

The Chinese remainder theorem allows us to restate
this result as follows.

\begin{thm*} [Hasse's Theorem: version 2]
If $F \in \Z[X_1, \ldots, X_m]$
is homogeneous of degree~$2$,
then $F(X_1, \ldots, X_m) = 0$ has a nontrivial $\Z$-solution
  if and only if
\begin{enumerate}
\item[(i)]  it has a nontrivial $\R$-solution, and
\item[(ii)] it has a primitive solution modulo $p^k$ for all
             primes~$p$ and exponents~$k\ge 1$.
\end{enumerate}
\end{thm*}

When a homogeneous Diophantine equation or system of such equations 
has primitive solutions modulo all 
powers $p^k$ of a given prime $p$,
we say that it is \emph{$p$-locally solvable}. 
If it is $p$-locally solvable for every prime $p$ and 
has nontrivial real solutions 
then we say that it is \emph{locally solvable}.
If it has a nontrivial $\Z$-solution then we say that it is
\emph{globally solvable}.
Global solvability clearly implies local solvability.
The above theorem states that, for a certain class of equations,
global solvability is actually equivalent to local solvability.

\begin{example} [$m=3$]
For equation~(\ref{EL}), 
Hasse's theorem is
a consequence of Legendre's theorem. See
the exercises at the end of this section.
\end{example}

\begin{example} [$m=2$]\label{example2}
The equation
$a X^2 + b X Y + c Y^2 = 0$ has a nontrivial $\Z$-solution if and only if
the discriminant $d = b^2 - 4 ac$ is a square, and has a primitive solution modulo 
an odd prime~$p$ if and only $d$ is a square modulo $p$.
So Hasse's theorem for $a X^2 + b X Y + c Y^2 = 0$  is a consequence of the
following theorem of Gauss:
\emph{If an integer is a square modulo $p$ for all odd primes $p$ then it is a square} \cite[Art.~125]{Gauss}. 

In this case one only needs to check modulo $p$ for odd primes $p$.
One does not need to check for solutions for $\R$, $p=2$, and powers of odd primes.
(The above theorem of Gauss extends to $n$th powers for $n$ up to $7$,
but it does not extend to $8$th powers
since $16$ is an $8$th power modulo all primes.\footnote{For a proof that this 
is in some sense the most general
counterexample, see Kraft and Rosen~\cite{KR}.}
As a consequence $X^8 - 16Y^8 = 0$ has primitive solutions modulo $p$ for
all primes $p$, but does not possess a global solution. 
The equation $X^8 - 16 Y^8 = 0$ is not, however, a counterexample to a higher-degree 
Hasse principle:
the solvability condition fails modulo $32$.)
\end{example}

Typically $p$-local solvability reduces to showing solvability modulo $p^k$ for
some sufficiently large $k$. As Example~\ref{example2} illustrates, 
and as we will see in this paper, $k=1$
is often enough. (See Appendix A for more
on the phenomenon of ``lifting'' solutions modulo $p$ to solutions modulo $p^k$.)

A natural setting for understanding solutions modulo~$p^k$ as
$k$ varies is through the
\emph{ring of $p$-adic integers~$\Z_p$} developed by Hensel.
Using the ring~$\Z_p$ allows one to organize a coherent
sequence of solutions modulo~$p^k$ for all $k$ into
one $p$-adic solution.
The field $\Q_p$ of $p$-adic numbers is the fraction field of~$\Z_p$.
The rings $\Z_p$ and fields~$\Q_p$ play a crucial role in modern
number theory, and are present in virtually every discussion of the
Hasse principle.
The current paper is somewhat exceptional:
in order to make this paper more accessible,
we do not use the $p$-adic numbers or Hensel's lemma.
We do discuss $\Z_p$ and Hensel's lemma in Appendix~A, but for now we merely
mention that, like $\R$, the field $\Q_p$ is complete for
a certain absolute value, and much of real or complex analysis
generalizes to~$\Q_p$. In fact, number theorists often formally
introduce a ``prime'' $\infty$, and denote $\R$ by~$\Q_\infty$;
the fields $\Q_p$ for $p$ a prime or $\infty$ give all the
\emph{completions} of~$\Q$ and are called the \emph{local fields}
associated with~$\Q$:

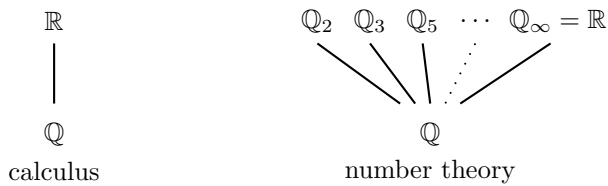
\begin{pspicture}(-3,-1)(2.5,2.2)
\rput(-2,1.5){$\R$}
\psline(-2,0.4)(-2,1.2)
\rput(-2,0){$\Q$}
\rput(-2,-0.5){calculus}
\rput(3,0){$\Q$}
\rput(1.5,1.5){$\Q_2$}
\psline(2.6,0.4)(1.5,1.2)
\rput(2.2,1.5){$\Q_3$}
\psline(2.8,0.4)(2.2,1.2)
\rput(2.9,1.5){$\Q_5$}
\psline(3,0.4)(2.9,1.2)
\rput(3.6,1.5){$\ldots$}
\psline[linestyle=dotted](3.2,0.4)(3.6,1.2)
\rput(4.7,1.5){$\Q_\infty = \R$}
\psline(3.4,0.4)(4.6,1.2)
\rput(3,-0.5){number theory}
\end{pspicture}

It is in this language that Hasse's theorem
achieves its standard form.

\begin{thm*} [Hasse's Theorem: version 3]
If $F \in \Z[X_1, \ldots, X_m]$
is homogeneous of degree~$2$,
then $F(X_1, \ldots, X_m) = 0$ has a nontrivial $\Q$-solution
if and only if
it has a nontrivial $\Q_p$-solution for all~$p$ (including $p=\infty$).
\end{thm*}

We say that a class of homogeneous equations
satisfies the \emph{Hasse principle} or
the \emph{local-global principle} if each equation
in the class has a nontrivial $\Z$-solution
if and only if ($i$) it has a nontrivial $\R$-solution, and
($ii$) it has a primitive solution modulo~$N$ for each $N$.
As mentioned above, the Chinese remainder theorem allows us
to replace ($ii$) by the following:
($ii'$) it has a primitive solution modulo~$p^k$ for each prime~$p$
and exponent~$k\ge 1$.
We formulate the Hasse principle for systems of homogeneous
polynomials in a similar manner.\footnote{We formulate the Hasse principle 
for homogeneous equations
in order to restrict our attention to integer solutions.
In the language of algebraic geometry, this formulation
asserts the existence of $\Q$-points (global solutions) on the associated projective 
variety given the existence of $\Q_p$-points (local solutions) for all~$p$
including $p=\infty$.}

However, \emph{the Hasse principle fails in general}.
In fact, it fails for the next
obvious class of equations: cubic equations.
The most famous
example is due to Selmer~\cite{Selmer}:
\begin{equation}\label{ESel}
    3X^3 + 4Y^3 + 5Z^3 = 0.
\end{equation}
This cubic obviously has nontrivial $\R$-solutions and
it can be shown to be locally solvable,
but \emph{it has no global solutions}.\footnote{Showing the absence of
global solutions is not elementary. Known proofs
of this fact use the arithmetic of cubic number fields; one possible approach
is to multiply (\ref{ESel}) through by $2$, change $Z$ to $-Z$, and factor the
left-hand side of the transformed
equation $6 X^3 + Y^3  = 10Z^3$ over $\Q(\sqrt[3]{6}\,)$. In contrast,
there are elementary proofs, like the one given in the current paper, that 
(\ref{ER}) has no global solutions.

The work of Selmer on this problem led Cassels to introduce the notion of Selmer groups
and to his groundbreaking work on Tate--Shafarevich groups in the theory
of elliptic curves; nowadays, Selmer's example can be interpreted as
representing an element of order $3$ in the Tate--Shafarevich group~$\sha_E$
of the elliptic curve $E: X^3 + Y^3 + 60Z^3 = 0$.
See \cite{Cassels}, \cite{Maz}, and our
Appendix~$B$
for more on the relationship between counterexamples
and the Tate--Shafarevich group.}

What if one sticks to quadratic equations, but allows \emph{systems} of
equations? In this paper we will show that the Hasse principle also
fails for this class. In particular we study systems of the form
\begin{equation}\label{twoquad}
a U^2 + b V^2 + c W^2 = d Z^2, \qquad UW= V^2,
\end{equation}
mainly when $b=0$,
and use
elementary
methods to produce counterexamples to the Hasse principle.
The case where $a=1, b=0, c=-17$, and $d=2$ is important
since it was the first known counterexample to the 
Hasse principle for Diophantine equations.\footnote{Counterexamples to the 
Hasse principle for norms, according to which an element of a number 
field is a norm if and only if it is a norm in every localization, were 
known for noncyclic extensions already around 1934.}
It was produced by Lind~\cite{Lind} and Reichardt~\cite{Rei} several 
years before Selmer's. As we will discuss in Section~\ref{s6}, 
the system (\ref{twoquad}) with  $a=1, b=0, c=-17$, and $d=2$
can be transformed into the single 
(nonhomogeneous) equation
\begin{equation}\label{ER}
    X^4 - 17Y^4 = 2 Z^2,
\end{equation}
the form considered by Lind
and Reichardt.

The purpose of this article is to give a self-contained,
accessible proof
of the existence of counterexamples
to the Hasse principle of the form (\ref{twoquad}),
counterexamples similar to
those of Lind and Reichardt's, using the easy
and well-known technique of parametrizing conics
to justify local solvability.\footnote{Other approaches 
use the less elementary method of
quartic Gauss and Jacobi sums.
Applied to quartics like $aX^4 + bY^4 = Z^2$
this method shows the solvability only for sufficiently large
values of $p$, and making the bounds explicit is quite technical.
Short but less elementary arguments can be given
by appealing to the Hasse-Weil bounds for curves of genus~$1$
defined over finite fields,
or F.~K.~Schmidt's result
on the existence of points on genus~$1$ curves over finite fields.}
The only required background is a standard undergraduate course in
number theory up to quadratic reciprocity, and a
standard undergraduate
course in modern algebra up to basic facts about polynomials
over rings and fields. (Material directed to a more
advanced audience will be confined to the footnotes
and the appendices.)
As far as we know, this paper is unique in developing interesting
counterexamples to the Hasse principle
in such an elementary manner.\footnote{A less interesting, but simpler
counterexample is $(X^2 - 2 Y^2)(X^2 - 17 Y^2)(X^2 - 34 Y^2)=0$.
To find solutions modulo~$p^k$, use properties of the Legendre symbol, and
Propositions~\ref{henselpower} and~\ref{fourthpowermod2}.}
We hope that
this paper will give a general mathematical audience a taste of this
interesting subject.

Variants of the Hasse principle, and the study of the manner in which
these principles fail, is a very important and active area of current
research.\footnote{For some accounts of recent activity 
see Mazur~\cite{Maz} or the summary of
mini-courses by Colliot-Th\'el\`ene and others 
(at
\texttt{http://swc.math.arizona.edu/oldaws/99GenlInfo.html}).}
As discussed in Appendix~B, these counterexamples are of interest
from the point of view of elliptic curves.


\bigskip\noindent
We conclude this section
by offering exercises showing the relationship between Legendre's
theorem, discussed at the start of this section,
and the Hasse principle.
As above, assume that $a, b, c\in\Z$ are such that $abc$ is  nonzero and
square-free.

\begin{exercise}
Let $p$ be a prime.
Call $(x_0, y_0, z_0)$ a \emph{$p$-focused} triple
if at most one of $x_0, y_0, z_0$ is divisible by~$p$.
Show that any primitive solution to the congruence
$$a X^2 + b Y^2 + c Z^2 \equiv 0 \pmod {p^2}$$
is $p$-focused.
\end{exercise}

\begin{exercise}
Suppose that $p\mid a$ and that
the congruence $a X^2 + b Y^2 + c Z^2 \equiv 0$ modulo $p$
has a $p$-focused solution. Show that $-bc$ is a square modulo~$p$.

Conclude that if $a X^2 + b Y^2 + c Z^2 \equiv 0$ modulo $p$
has a $p$-focused solution for all odd $p\mid a$, then
$-bc$ is a square modulo~$|a|$.
\end{exercise}

\begin{exercise}
Take Legendre's theorem as given
and use the preceding exercise to show that
if the congruence $a X^2 + b Y^2 + c Z^2 \equiv 0$ modulo $p$
has a $p$-focused solution for all odd primes $p\mid abc$,
and if the equation $a X^2 + b Y^2 + c Z^2 =0$
has a nontrivial $\R$-solution, then
$a X^2 + b Y^2 + c Z^2 =0$ has a
nontrivial $\Z$-solution.
\end{exercise}

\begin{exercise}
Use the above exercises to show that the Hasse
principle for the equation $a X^2 + b Y^2 + c Z^2 =0$
is a consequence of Legendre's theorem.
\end{exercise}


\section{Parametrizing Conics} \label{s2}

A standard method for finding Pythagorean triples
is through the rational para\-metriza\-tion
of the unit circle~$\big\{ (x,y) \in\R^2 \mid x^2+y^2 = 1\big\}$.
The parametrization is found
by intersecting the circle with the line of slope $t$ going
through the point $P = (-1, 0)$ of the circle.
The line defined by $y=t(x+1)$ intersects
the circle defined by $x^2+y^2 - 1 = 0$
at points whose first coordinates
satisfy the equation $0=x^2+t^2(x+1)^2-1= (x+1)(x-1 + t^2x + t^2)$.
Thus, the points of intersection are the point $P = (-1,0)$ we
started with, as well as
$P_t = \bigl( \frac{1-t^2}{1+t^2}, \frac{2t}{1+t^2} \bigr)$.

\smallskip

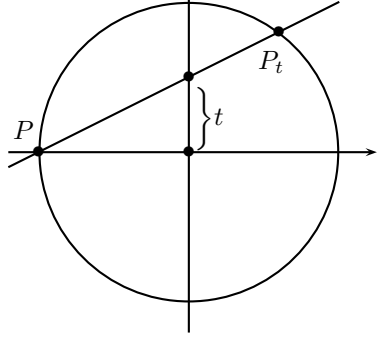
\begin{figure}[ht]
\begin{pspicture}(0,-2)(2,2)
\pscircle(0,0){2}
\psline(-2.4,-0.2)(2,2)
\psline{->}(-2.4,0)(2.5,0)
\psline{->}(0,-2.4)(0,2.5)
\rput(-2.2,0.3){$P$}
\rput(1.1,1.2){$P_t$}
\rput(0,0){$\bullet$}
\rput(0,1){$\bullet$}
\rput(0.27,0.45){$\bigg\} t$}
\rput(1.2,1.6){$\bullet$}
\rput(-2,0){$\bullet$}
\end{pspicture}
\caption{Parametrizing the unit circle.}
\end{figure}

This parametrization leads us to the following identity in $\R[T]$:
\begin{equation}\label{standard}
(1-T^2)^2 + (2T)^2 = (1+T^2)^2.
\end{equation}
Specializing $T$ to $n/m$ with $n,m\in\Z$
gives Pythagorean triples:
$$  (m^2-n^2)^2 + (2mn)^2 = (m^2+n^2)^2.  $$

The above procedure is purely algebraic, and
there is no problem modifying it to the
equation
$ax^2 + by^2 = 1$
over a general field~$F$ where $a,b\in F$ are nonzero.
Of course, we need  a starting point:
we need $x_0, y_0\in F$ such that $ax_0^2 + b y_0^2 = 1$.
The analogue to~(\ref{standard}) is displayed in the following lemma
as~(\ref{qequation}).\footnote{In
the language of algebraic geometry, a nonsingular plane conic
possessing at least one $F$-rational point
is isomorphic to ${\mathbb P}^1$ via such a parametrization.
The restriction to conics of the form $ax^2+by^2=1$
is not a true restriction:
if $\Char F\ne 2$ then every nondegenerate conic
can be brought into the form $ax^2+by^2=1$ with a projective transformation.}

\begin{lemma}\label{qlemma}
Let $F$ be a field, and let $a,b\in F$ be nonzero.
Let $x_0, y_0 \in F$ be such that $ax_0^2+b y_0^2 =1$.
Then in $F[T]$
\begin{equation}\label{qequation}
a q_1^2  + b q_2^2 =  q_3^2
\end{equation}
where
$$q_1 = b x_0 T^2 - 2  b y_0 T - a x_0,\quad
q_2 = - b y_0 T^2 - 2 a  x_0 T + a y_0,  \quad q_3 = b T^2 + a.$$
Furthermore, at least two of $q_1, q_2, q_3$ have degree exactly $2$.
Finally, if $\Char F \ne 2$, then each of $q_1, q_2, q_3$ is nonzero,
and no two are associates.\footnote{Recall
that two nonzero polynomials of $F[T]$ are associates if one
is a constant multiple of the other. More generally, in any unique factorization domain, two nonzero elements are called \emph{associates} if the second is the product of a unit with the first.}
\end{lemma}

\begin{remark}
The polynomials $q_1, q_2, q_3$ are found using the parametrization method, 
but how they are discovered is not crucial to the proof below. What is important 
is that $a q_1^2 + b q_2^2 =  q_3^2$.
\end{remark}

\begin{proof}
A straightforward calculation verifies that $a q_1^2 + b q_2^2 =  q_3^2$.
Observe that $\deg q_3 =2$ since $b\ne 0$.
Since $q_3^2 = a q_1^2 + b q_2^2$, we also have
$\deg q_1 =2$ or $\deg q_2 =2$.

Assume $\Char F\ne 2$.
Since $a$ and $b$ are nonzero, and $x_0$ and $y_0$ are not both~$0$,
each of $q_1, q_2, q_3$ is nonzero.
Suppose two of $q_1, q_2, q_3$ are associates. Then
these two must have degree $2$.
The equation $a q_1^2 + b q_2^2 =  q_3^2$ then implies
$q_1^2, q_2^2, q_3^2$ are all associates.
By Lemma~\ref{assoclemma} below,
$q_1$ and $q_2$ are constant multiplies of $q_3$.
But this contradicts the fact that at least one of $q_1$ or $q_2$
has a nonzero linear term.
\end{proof}

The above makes use of the following general fact:

\begin{lemma}\label{assoclemma}
Let $a$ and $b$ be nonzero elements of a unique factorization domain~$R$.
If $a^n$ and $b^n$ are associates for some positive~$n$, then $a$ and $b$ are associates.
\end{lemma}

\begin{proof}
Factor $a$ and $b$ into irreducible elements. Up to multiplication by units, $a^n$ and $b^n$ have the same factorization.
This means that $a$ and $b$ must also have the same factorization up to multiplication by units.
\end{proof}

The existence of $q_1, q_2, q_3$ in the above lemma depends
on the existence of at least one solution $a x_0^2+b y_0^2 = 1$.
For $F= \F_p$ the existence of such a solution
follows from Euler's criterion.\footnote{Suppose $p$ is an odd prime
and $a\in \F_p$ is nonzero.
Euler's criterion
states that $a$ is a square in
the field~$\F_p$ if
and only if $a^{(p-1)/2} = 1$, and that $a$ is not a square in $\F_p$
if and only if $a^{(p-1)/2} = -1$.
Euler's criterion is a consequence of
the well-known result that the multiplicative group  of nonzero
elements of $\F_p$ is cyclic of order~$p-1$. 

In fact, the multiplicative group $F^\times$ of any finite field~$F$ is cyclic.
Thus Euler's criterion generalizes to any finite field $F$ of odd order.
One can also use the cyclic nature of $F^\times$ to show that
every element of a finite field of even order is a square.
As a consequence, Lemma~\ref{conicpoint} generalizes
to arbitrary finite fields.}

\begin{lemma}\label{conicpoint}
Let $a$ and $b$ be nonzero elements of the field $\F_p$ where
$p$ is a prime. Then there exist $x_0, y_0\in \F_p$ such that
$a x_0^2+b y_0^2 = 1$.
\end{lemma}

\begin{proof}
If $p=2$, take $x_0=1$ and $y_0=0$.
If $p>2$,
we wish to solve $y^2=f(x)$ where $f(x)=b^{-1}(1-a x^2)$.
If there are no solutions, then  $f(t)$ is a nonsquare
for each $t\in \F_p$. By Euler's criterion,
  $f(t)^{(p-1)/2} = -1$
for all $t\in \F_p$. However,
this contradicts the fact that the
degree-($p-1$) polynomial $f(x)^{(p-1)/2} + 1$ has at
most $p-1$ roots.
\end{proof}

\begin{remark}
This result is due to Euler and arose
from his attempt to prove Fermat's claim that every
integer is the sum of four squares, a theorem
finally proved by Lagrange (see \cite[Chapter III \S XI]{Weil}).
For that theorem,
one uses the solvability
of  $-x^2 - y^2 = 1$ in $\F_p$.
\end{remark}

\begin{remark}
As we have seen, methods for parametrizing the unit circle
turn out to apply to conics defined over general fields.
This is an example of a
major theme of arithmetic geometry, that many ideas of geometry
carry over to other fields of interest to number theorists.
The reader might be amused to see
a further example.
Figure \ref{F1} displays the plane over~$\F_7$, which has
$7^2$ points, denoted by $+$ or $\bullet$, and the unit circle $x^2 + y^2 = 1$,
consisting of $8$ points, denoted by $\bullet$. In the graph
on the right, the line $L: y = x+3$ is displayed. Note that every line
in the affine plane over $\F_7$ contains $7$ points -- the segment
between, say, $(0,3)$ and $(1,4)$ is not part of the line: it is
only drawn to help us visualize the line. We also remark
that $L$ can also be written as $y = -6x+3$; this may change how we
draw segments between the points but, of course, not the
points on $L$. The line $L$ intersects the unit circle in exactly one
point, namely $(2,5)$, and hence is the tangent to the unit circle at
this point. Similarly, $x = 1$ is the tangent to the unit circle at
$(1,0)$. The line $y = 2x-1$ intersects the unit circle in exactly
two points: can you see which?

We leave it as an exercise to the reader to determine the interior
of the unit circle: these are defined to be the points that do not lie on any
tangent to the circle. For example, the points $(1,x)$ with
$1 \le x \le 6$ are exterior points since they lie on the tangent
at $(1,0)$.
\end{remark}

\begin{figure}[ht]
\begin{picture}(120,120)(30,0)
\put(-3.8,-2.5){$+$}
\put(-2.5,17.5){$\bullet$}
\put(-3.8,37.5){$+$}
\put(-3.8,57.5){$+$}
\put(-3.8,77.5){$+$}
\put(-3.8,97.5){$+$}
\put(-2.5,117.5){$\bullet$}
\put(17.5,-2.5){$\bullet$}
\put(16.2,17.5){$+$}
\put(16.2,37.5){$+$}
\put(16.2,57.5){$+$}
\put(16.2,77.5){$+$}
\put(16.2,97.5){$+$}
\put(16.2,117.5){$+$}
\put(36.2,-2.5){$+$}
\put(36.2,17.5){$+$}
\put(37.5,37.5){$\bullet$}
\put(36.2,57.5){$+$}
\put(36.2,77.5){$+$}
\put(37.5,97.5){$\bullet$}
\put(36.2,117.5){$+$}
\put(56.2,-2.5){$+$}
\put(56.2,17.5){$+$}
\put(56.2,37.5){$+$}
\put(56.2,57.5){$+$}
\put(56.2,77.5){$+$}
\put(56.2,97.5){$+$}
\put(56.2,117.5){$+$}
\put(76.2,-2.5){$+$}
\put(76.2,17.5){$+$}
\put(76.2,37.5){$+$}
\put(76.2,57.5){$+$}
\put(76.2,77.5){$+$}
\put(76.2,97.5){$+$}
\put(76.2,117.5){$+$}
\put(96.2,-2.5){$+$}
\put(96.2,17.5){$+$}
\put(97.5,37.5){$\bullet$}
\put(96.2,57.5){$+$}
\put(96.2,77.5){$+$}
\put(97.5,97.5){$\bullet$}
\put(96.2,117.5){$+$}
\put(117.5,-2.5){$\bullet$}
\put(116.2,17.5){$+$}
\put(116.2,37.5){$+$}
\put(116.2,57.5){$+$}
\put(116.2,77.5){$+$}
\put(116.2,97.5){$+$}
\put(116.2,117.5){$+$}
\end{picture}
\begin{picture}(120,120)(-20,0)
\put(-3.8,-2.5){$+$}
\put(-2.5,17.5){$\bullet$}
\put(-3.8,37.5){$+$}
\put(-3.8,57.5){$+$}
\put(0,60){\line(1,1){60}}
\put(-3.8,77.5){$+$}
\put(-3.8,97.5){$+$}
\put(-2.5,117.5){$\bullet$}
\put(17.5,-2.5){$\bullet$}
\put(16.2,17.5){$+$}
\put(16.2,37.5){$+$}
\put(16.2,57.5){$+$}
\put(16.2,77.5){$+$}
\put(16.2,97.5){$+$}
\put(16.2,117.5){$+$}
\put(36.2,-2.5){$+$}
\put(36.2,17.5){$+$}
\put(37.5,37.5){$\bullet$}
\put(36.2,57.5){$+$}
\put(36.2,77.5){$+$}
\put(37.5,97.5){$\bullet$}
\put(36.2,117.5){$+$}
\put(56.2,-2.5){$+$}
\put(56.2,17.5){$+$}
\put(56.2,37.5){$+$}
\put(56.2,57.5){$+$}
\put(56.2,77.5){$+$}
\put(56.2,97.5){$+$}
\put(56.2,117.5){$+$}
\put(80,0){\line(1,1){40}}
\put(76.2,-2.5){$+$}
\put(76.2,17.5){$+$}
\put(76.2,37.5){$+$}
\put(76.2,57.5){$+$}
\put(76.2,77.5){$+$}
\put(76.2,97.5){$+$}
\put(76.2,117.5){$+$}
\put(96.2,-2.5){$+$}
\put(96.2,17.5){$+$}
\put(97.5,37.5){$\bullet$}
\put(96.2,57.5){$+$}
\put(96.2,77.5){$+$}
\put(97.5,97.5){$\bullet$}
\put(96.2,117.5){$+$}
\put(117.5,-2.5){$\bullet$}
\put(116.2,17.5){$+$}
\put(116.2,37.5){$+$}
\put(116.2,57.5){$+$}
\put(116.2,77.5){$+$}
\put(116.2,97.5){$+$}
\put(116.2,117.5){$+$}
\end{picture}
\caption{The unit circle over $\F_7$, and its tangent at $(2,5)$.}\label{F1}
\end{figure}
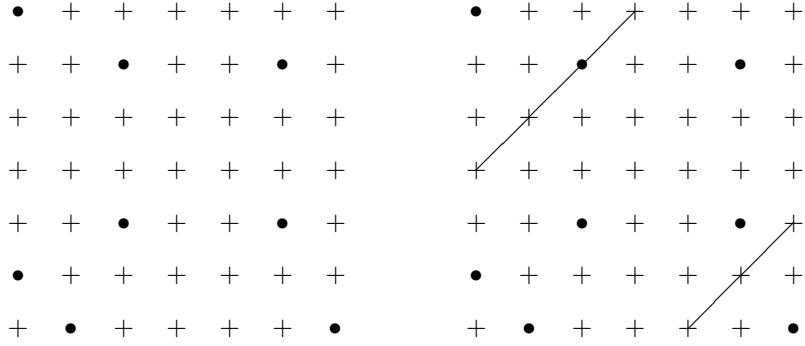


\section{Solutions modulo odd primes} \label{s3}

Let $p$ be an odd prime.
We now turn our attention to finding nontrivial solutions 
of the system
\begin{equation*}
a U^2 + c W^2 = d Z^2, \qquad UW= V^2
\end{equation*}
with values in $F=\F_p$.
Here $a, c, d\in \F_p$ are assumed nonzero.

Replacing $a$ and $c$ with $a d^{-1}$ and $c d^{-1}$
we can assume $d=1$.
The first equation $a U^2 + c W^2 =  Z^2$  can be parametrized as in
Section~\ref{s2}. This gives a family of solutions to the first equation, and
we need to determine that at least
one of these solutions also satisfies~$UW=V^2$.
The following is a key ingredient to doing so.\footnote{Everything in this section extends easily to
finite fields of odd order, not just prime fields.}

\begin{lemma}\label{PA}
Let $f, g \in \F_p[X]$ be nonzero
polynomials of degree at most two.
If $\big(\frac{f(t)}{p}\big) = \big(\frac{g(t)}{p}\big)$
for all $t \in \F_p$, or
if $\big(\frac{f(t)}{p}\big) = - \big(\frac{g(t)}{p}\big)$ for all $t 
\in \F_p$,
then $f$ and $g$ are associates.
\end{lemma}

\begin{remark}
The statement of this result uses the \emph{Legendre symbol} $\displaystyle\left(\frac{a}{p}\right)$.
If $a \in \F_p$ then the Legendre symbol can be defined as follows: 
\[
 \left(\frac{a}{p}\right) \, = \,
 \begin{cases}
 	\phantom{-}1 			&\text{if $a$ is a nonzero square in~$\F_p$,}\\
	-1			&\text{if $a$ is not a square in~$\F_p$,}\\
 	\phantom{-}0	\quad 	&\text{if $a=0$  in~$\F_p$.}
 \end{cases}
\]
The Legendre symbol provides a convenient notation for the expression of theorems such as quadratic
reciprocity or the simpler theorem that $\big(\frac{a b}{p}\big) = \big(\frac{a }{p}\big) \big(\frac{b}{p}\big)$.
This last equation is a consequence of
Euler's criterion:
$$
a^{(p-1)/{2}} \, = \, \left(\frac{a}{p}\right),
$$
where we view the values $0, 1, -1$ of the Legendre symbol as 
elements of~$\F_p$. Recall $p$ is an odd prime, so $0, 1, -1$ are
distinct.
\end{remark}

\begin{proof}
By Euler's criterion, $\big(\frac{f(t)}{p}\big)= f(t)^{(p-1)/2}$
and $\big(\frac{g(t)}{p}\big)= g(t)^{(p-1)/2}$.
So if $\big(\frac{f(t)}{p}\big) = \big(\frac{g(t)}{p}\big)$ for all 
$t \in \F_p$,
then every $t\in\F_p$ is a root of $f^{(p-1)/2}- g^{(p-1)/2}$.
Recall that a nonzero polynomial in $\F_p[T]$
of degree $d$ has at most $d$ roots since $\F_p$ is a field.
The polynomial 
$f^{(p-1)/2}- g^{(p-1)/2}$ has degree at most $p-1$, but has
$p$ roots.  We conclude that
 $f^{(p-1)/2}- g^{(p-1)/2}$ is the zero polynomial.

Since $\F_p[X]$ is a
unique factorization domain, and since $f^{(p-1)/2} = g^{(p-1)/2}$,
the polynomials $f$ and $g$ are associates by Lemma~\ref{assoclemma}.

If $\big(\frac{f(t)}{p}\big) = - \big(\frac{g(t)}{p}\big)$
for all $t \in \F_p$, pick a nonsquare $r$ in $\F_p$.
Observe that
$\big(\frac{f(t)}{p}\big) = \big(\frac{r g(t)}{p}\big)$ for all $t \in \F_p$,
so $f$ and $rg$ are associates by the conclusion of the previous case.
In particular, $f$ and $g$ are associates.
\end{proof}

\begin{thm}\label{T1}
Let $p$ be an odd prime and let $a, c, d\in \F_p$ be nonzero.
Then the system
\begin{equation*}
a U^2 + c W^2 = d Z^2, \qquad UW= V^2
\end{equation*}
has a nontrivial $\F_p$-solution.
\end{thm}

\begin{proof}
As mentioned above, we reduce to the case~$d=1$.
Parametrizing the conic $a U^2 + c W^2 =  Z^2$
as in Lemma \ref{qlemma} (using Lemma~\ref{conicpoint})
yields nonzero polynomials
 $q_1, q_2, q_3 \in \F_p[T]$ where
 $a q_1^2 + c q_2^2 = q_3^2$ and where $q_1$ and $q_2$ are not associates.

By Lemma~\ref{PA} and the fact that $q_1$ and $q_2$ are not associates,
there is a $t \in \F_p$ such that
$\big(\frac{q_1 (t)}{p}\big) \ne - \big(\frac{q_2(t)}{p}\big)$.
So $q_1(t)$ and $q_2(t)$ are not both zero
and $\big(\frac{q_1 (t) q_2(t)}{p}\big) \ne - 1$.
Thus $q_1 (t) q_2(t) = v^2$ for some $v\in\F_p$.
Hence $U=q_1(t), W= q_2(t), Z=q_3(t)$, $V=v$ is a nontrivial solution.
\end{proof}


\section{Solutions modulo prime powers} \label{s4}

Now we focus  on the $p$-local solvability of the system
\begin{equation}\label{msys}
a U^2 + c W^2 = d Z^2, \qquad UW= V^2
\end{equation}
where $a, c, d$ are nonzero integers.
A consequence of Theorem~\ref{T1} is that
if $p$ is an odd prime and if $p\nmid acd$
then this system
has a primitive solution modulo $p$. In this section we extend this result to
powers of~$p$. We also discuss the case of $p=2$. 

For the convenience of the reader, we state and prove the following
well-known result from elementary number theory.

\begin{prop} \label{henselpower}
Let $p$ be a prime and let $N$ and $r>0$ be integers such that
$p \nmid  r N$.
If $N$ is an $r$th power modulo $p$, then $N$ is an $r$th power
  modulo $p^{k}$ for all~$k\ge 1$.
\end{prop}

\begin{proof}
We proceed by induction on~$k$.
Suppose that $N\equiv a^r$ modulo $p^k$. Write $N = a^r + c p^k$, and
let $x$ be a solution to $r a^{r-1}x \equiv c$ modulo $p$.
Using the binomial expansion,
$$\left(a+xp^k\right)^r \equiv a^r + ra^{r-1} x p^k
\equiv a^r + c p^k  \pmod {p^{k+1}}.$$
\end{proof}

We now consider the main result of this section. 
A \emph{strong solution} modulo~$p^k$ to the system~(\ref{msys}) is a primitive 
solution $(u, v, w, z)$
to the associated congruences modulo~$p^k$ such that 
at least one of $au, cw, dz$ is nonzero modulo the prime~$p$. 

\begin{thm} \label{strongliftingtheorem}
Let $p$ be an odd prime, and let $a, c, d$ be nonzero integers. 
If 
\begin{equation*}
a U^2 + c W^2 = d Z^2, \qquad UW= V^2
\end{equation*}
has a strong solution modulo $p$, then it has a strong solution modulo $p^k$
for all~$k$. In particular, it is $p$-locally solvable.
\end{thm}

\begin{proof}
Let $(u_0, v_0, w_0, z_0)$ be a strong solution modulo $p$. Since $a u_0^2 + c w_0^2 = d z_0^2$,
at least two of $a u_0, c w_0, d z_0$ must be nonzero modulo $p$. 
By symmetry we can assume that $au_0$
is nonzero modulo $p$. Fix a power $p^k$ of $p$. 
Since $p\nmid a$ and $p\nmid u_0$, we can choose inverses
$a^{-1}, u_0^{-1} \in \Z$ modulo $p^k$.

Let $v = v_0 \, u_0^{-1},  \; w = w_0 \, u_0^{-1},$ and $z = z_0\,  u_0^{-1}$. 
Then $(1, v, w, z)$ also solves the system modulo~$p$. 
So $w \equiv v^2$ modulo $p$, and hence $a + c v^4 \equiv dz^2$ modulo $p$.
Since $a^{-1} \left( d z^2 - c v^4\right) \equiv 1$ modulo $p$, and since $1$ is a fourth power,
Proposition~\ref{henselpower} guarantees the existence of an $m\in \Z$
such that $a^{-1} \left( d z^2 - c v^4\right) \equiv m^4$ modulo~$p^k$.
In other words, $a m^4 + c v^4 \equiv d z^2$ modulo~$p^k$, so
$(m^2, m v, v^2, z)$ is a solution modulo~$p^k$.
It is a strong solution since $m$ is nonzero modulo~$p$.
\end{proof}

Theorem~\ref{strongliftingtheorem} and Theorem~\ref{T1}
yield the following:

\begin{cor} \label{plocalcor}
If $p$ is an odd prime such that $p\nmid acd$
then the system
\begin{equation*}
a U^2 + c W^2 = d Z^2, \qquad UW= V^2
\end{equation*}
is $p$-locally solvable: it has primitive  solutions modulo $p^k$ for all~$k$.
\end{cor}

For $p=2$ the situation is more subtle. For example, the system
\begin{equation*}
U^2 + 3 W^2 = 7 Z^2, \qquad UW= V^2
\end{equation*}
has solution $(1, 1, 1, 2)$ modulo $2$. In fact, $(1, 1, 1, 2)$ is a solution modulo $2^3$.
However, the system has no primitive solution modulo $2^4$.
The following exercise provides a shortcut for verifying the
nonexistence of solutions modulo $16$.

\begin{exercise}\label{exerciseonmod16}
Consider the system $aU^2 + c W^2 = d Z^2, \; \; UW= V^2$ where $a, c, d \in\Z$ are odd.
Show that if there is a primitive solution
modulo $16$, then there is a solution $(u, v, w, z)$ with $u, v, w\in\{0, 1\}$ and $z\in\{0, 1, 2, 3\}$.
\end{exercise}

A difficulty with extending Theorem~\ref{strongliftingtheorem}
to $p=2$ is the failure of Proposition~\ref{henselpower}
to generalize: the integer $3$ is a fourth power
modulo $p=2$ but not modulo~$2^k$ if $k>1$. The following provides the needed
 variant to Proposition~\ref{henselpower}.

\begin{prop}\label{fourthpowermod2}
If $N \equiv 1$ modulo $2^4$, then $N$ is a fourth power modulo $2^k$
for all~$k\ge1$.
\end{prop}

\begin{proof}
We use induction for~$k\ge 4$.
Suppose $N\equiv a^4$ modulo~$2^k$
where $k\ge 4$. Write $N = a^4 + c 2^k$.
Using the binomial expansion and the fact that $a^3\equiv 1$ modulo $2$,
$$(a+c 2^{k-2})^4 \equiv a^4 + 4a^{3} c 2^{k-2}  \equiv a^4 + c 2^k  \pmod {2^{k+1}}.$$
\end{proof}

Using this we can prove the following. Its proof is similar to that of Theorem~\ref{strongliftingtheorem}.

\begin{thm} \label{strongliftingtheorem2}
Let $a, c, d$ be nonzero integers. 
If the system
\begin{equation*}
a U^2 + c W^2 = d Z^2, \qquad UW= V^2
\end{equation*}
has a strong solution modulo $2^4$, then it has a strong solution modulo $2^k$
for all~$k$.
\end{thm}

In some cases there is no distinction between primitive
and strong solutions:

\begin{lemma}\label{lemma.primitive.strong}
Suppose $p$ is a prime, $k\ge 2$, and $a, c, d\in\Z$ are such that $p^2\nmid acd$.
Then every primitive solution to $(\ref{msys})$ modulo $p^k$ is a strong solution.
\end{lemma}

\begin{proof}
Let $(u, v, w, z)$ be a primitive solution modulo~$p^k$
that is not a strong solution. We derive a contradiction in the case where $p\nmid u$; the other
cases are similar. Since $(u, v, w, z)$ is not a strong solution, $p\mid a$.
Thus $p\nmid cd$. So $p$ divides both $w$ and~$z$ since $(u, v, w, z)$ is not a strong solution.
Looking at $a u^2 + c w^2 \equiv d z^2$ modulo~$p^2$ gives us that
$a u^2 \equiv 0$ modulo $p^2$. Thus $p^2\mid a$, a contradiction.
\end{proof}

\begin{cor}\label{characterizationeasycase}
Let $p$ be a prime, and let $a, c, d$ be integers such that $p^2 \nmid acd$.
If $p$ is odd then the system $(\ref{msys})$ is $p$-locally solvable if 
and only if it possesses a strong solution modulo $p$. If $p=2$, the 
system is $p$-locally solvable if and only if it possesses a strong 
solution modulo $2^4$.
\end{cor}

\begin{proof}
One direction follows from Theorems~\ref{strongliftingtheorem} and~\ref{strongliftingtheorem2}.
For the other direction, suppose (\ref{msys}) is $p$-locally solvable.
So there is a primitive solution modulo~$p^4$.
By Lemma~\ref{lemma.primitive.strong}, there is a strong solution modulo~$p^4$. 
For the case where $p$ is odd, observe that such a solution is also a strong solution modulo~$p$.
\end{proof}

The general case where $p^2$ is allowed to divide $acd$ will be considered in 
Section~\ref{s6}.


\section{Counterexamples to the Hasse principle}\label{s5}

The goal of this section is to identify counterexamples to the Hasse principle.  These
are systems that are locally solvable but not globally solvable: they
lack
nontrivial $\Z$-solutions.

We start with the question of local solvability.

\begin{prop}\label{locallysolvable}
Suppose
\begin{enumerate}
\item
$q$ and $d$ are relatively prime nonzero integers and $q$ is positive,
\item
$q \equiv 1$ modulo $16$,
\item
$d$ is a square modulo $p$ for all primes $p\mid q$, and
\item
$q$ is a fourth power modulo~$p$ for all odd primes $p\mid d$.
\end{enumerate}
Then the following system is locally solvable:
\begin{equation*}
U^2 - q  W^2 = d Z^2, \qquad UW= V^2
\end{equation*}
\end{prop}

\begin{proof}
Observe that $(u, v, w, z)=\left(q^{1/2}, q^{1/4}, 1, 0\right)$ is a real solution.
For all $p\nmid 2 q d$, the system is $p$-locally solvable by Corollary~\ref{plocalcor}.
Observe that $(u, v, w, z)=\left(1, 1, 1, 0\right)$
is a strong  solution modulo 16.
By Theorem~\ref{strongliftingtheorem2}, the system is $2$-locally solvable.

Suppose $p\mid q$ (so $p$ is odd). Let $m$ be such that $m^2\equiv d$ modulo $p$.
Then $(u, v, w, z)=(m, 0, 0, 1)$ is a solution modulo $p$. Since $d$ and $q$ are
relatively prime, $p\nmid d$. Thus the solution is strong.
By Theorem~\ref{strongliftingtheorem}, the system is $p$-locally solvable.

Suppose $p\mid d$ is odd. Let $m$ be such that $m^4\equiv q$ modulo $p$.
Then $(u, v, w, z)=(m^2,m, 1, 0)$ is a solution modulo $p$. 
The solution is strong since $p\nmid q$.
By Theorem~\ref{strongliftingtheorem}, the system is $p$-locally solvable.
\end{proof}

Now we will find systems that have no nontrival $\Z$-solutions. 
Our examples will rely on the following key lemma:

\begin{lemma}\label{redZR} 
Let $d$ be a nonzero square-free integer, and let $q\equiv 1$ modulo $8$ 
be a prime not dividing $d$.
If the system $U^2 - q  W^2 = d Z^2, \;\; UW= V^2$
has a nontrivial $\Z$-solution, 
then $d$ is a fourth power modulo~$q$.
\end{lemma}

\begin{proof}
Since the system is homogeneous, it has a primitive solution $(u, v, w, z)$. 
Observe that
$u$ and $w$ must be relatively prime since $d$ is square-free: if $p\mid u$ 
and $p\mid w$ then $p\mid v$ and $p^2 \mid d z^2$, so $p^2 \mid d$, a contradiction.
A similar argument shows that $u$ and $z$ are relatively prime, and $w$ and $z$ are relatively prime.
Also, since $u^2 w^2 =v^4$ and $u, w$ are relatively prime,  $u^2$ and $w^2$ must be
fourth powers.

Let $p$ be an odd prime dividing $z$. 
Modulo $p$ we have $u^2 \equiv q w^2$. Let $w^{-1}$ be an inverse
to $w$ modulo~$p$. So $q \equiv \left( u w^{-1} \right)^2$ modulo $p$.
In terms of the Legendre symbol: $(\frac q p)  = 1$.
Since $q\equiv 1$ modulo $4$, quadratic reciprocity tells us that
$(\frac p q)=1$.

Thus $(\frac p q)=1$ for all odd $p\mid z$. Since $q\equiv 1$ modulo $8$ 
we also have $(\frac 2 q)=1$ and  $(\frac {-1} q)=1$. 
By the multiplicativity of the Legendre symbol it follows that 
$(\frac z q)=1$. Thus $z^2$ is a nonzero fourth power modulo $q$.

We have that $u^2 \equiv d z^2$ modulo $q$. We know that $u^2$ and $z^2$
are fourth powers modulo~$q$. It follows that $d$ is a fourth power 
modulo $q$.
\end{proof}

Now consider the system of homogeneous Diophantine equations
\begin{equation}\label{system4}
U^2 - q  W^2 = d Z^2, \qquad UW= V^2
\end{equation}
where
\begin{enumerate}
\item $q$ is a prime such that $q \equiv 1$ modulo $16$,
\item $d$ is nonzero, square-free, and not divisible by $q$,
\item $d$ is a square, but not a fourth power,  modulo~$q$, and
\item $q$ is a fourth power modulo~$p$ for every odd $p$ dividing $d$.
\end{enumerate}

Proposition~\ref{locallysolvable} and
Lemma~\ref{redZR} together gives us our main result.

\begin{thm}\label{system4thm}
The system~\emph{(\ref{system4})} is locally solvable but not globally solvable: it has no nontrivial
$\Z$-solutions. 
\end{thm}

We end with a few specific examples of~(\ref{system4}).

\begin{example}
Lind and Reichardt's example, the first known counterexample
to the Hasse principle, is the
following special case of Theorem~\ref{system4thm}:
$$
U^2 - 17  W^2 = 2 Z^2, \qquad UW= V^2.
$$
This is a counterexample since $2$ is a square, but not a fourth power,
modulo~$17$.
\end{example}

\begin{example}
More generally, let $q$ be a prime such that $q \equiv 1$ modulo $16$
and such that $2$ is not a fourth power modulo~$q$.
In this case  $( \frac{2}{q} ) = 1$, so
$$ U^2 - q  W^2 = 2 Z^2, \qquad UW= V^2 $$
gives a counterexample to the Hasse principle.\footnote{Actually,
we need only assume $q \equiv 1$ modulo $8$
since $(1, 1, 1, 2)$ is a strong solution modulo~$16$ if $q\equiv 9$ modulo $16$.
Also we note that
this family of examples is infinite.
In fact, the density
of primes $q$ with $q\equiv 1$ modulo $8$ such that $2$
is not a fourth power modulo~$q$ is $1/8$.
This can be seen by applying the Chebotarev density theorem to
the extension $\Q(\sqrt[4]{2}, i)/\Q$, whose Galois group
is the dihedral group of order~$8$.}
\end{example}

\begin{example}
For an example where $d\ne 2$, consider
$$
U^2 - 17  W^2 = 19 Z^2, \qquad UW= V^2.
$$
\end{example}


\section{Further Issues}\label{s6}

We conclude by addressing two issues raised earlier. The first concerns 
the extension of Corollary~\ref{characterizationeasycase} to the case where $p^2 \mid abc$.
The second concerns the relationship between the systems studied in this
paper and the Diophantine equation $a X^4 + b X^2 Y^2 + c Y^4 = d Z^2$ studied by others.

\subsection*{Local Solvability}
Corollary~\ref{characterizationeasycase} gives necessary and sufficient 
conditions for the \hbox{$p$-local} solvability of
the system \hbox{$a U^2 + c  W^2 = d Z^2$}, $\;UW= V^2$
when $p^2 \nmid abc$. These conditions give a computationally effective
procedure for deciding $p$-local solvability.
We will now discuss an effective procedure for deciding $p$-local 
solvability even when $p^2 \mid abc$. With Corollary~\ref{plocalcor} 
this implies that local solvability as a whole is effectively decidable 
(deciding $\R$-solvability is easy).

We begin with some terminology. Given nonzero integers $a, c, d$ we 
refer to \hbox{$a U^2 + c  W^2 = d Z^2$}, $\;UW= V^2$
as the \emph{system~$(a, c, d)$}. We write 
$(a,c,d) \sim (a',c',d')$ if the two systems are both 
$p$-locally solvable or neither is. 

\begin{lemma} \label{fs}
Let $a, c, d\in\Z$ be nonzero integers and let $p$ be a prime.
Then 
\begin{align*}
(a,   c,  d) & \sim (c,   a,  d) \sim (pa, pc, pd) \sim 
\bigl( a p^2, c p^2, d     \bigr) \\
 & \sim \bigl( a,     c,     d p^2 \bigr) \sim
   \bigl( a p^4, c,     d     \bigr) \sim \bigl( a,     c p^4, d     \bigr).
\end{align*}
\end{lemma}

\begin{proof}
We show that the $p$-local sovability of $(a, c, d)$ implies that of 
$\bigl( a p^4, c, d \bigr)$. The proofs of the other implications are similar.

Suppose $(u, v, w, z)$ is a primitive solution
modulo~$p^{k}$ to the system $(a, c, d)$.
Then  $( u, p v, p^2 w,  p^2 z)$ is a solution modulo $p^{k+2}$ to
the system $\bigl( a p^4, c, d \bigr)$. If this solution is not primitive,  
divide each coordinate 
by the largest common power of~$p$ 
(either $p$ or $p^2$) to obtain a primitive solution modulo $p^{k-2}$ 
to the system~$\bigl( a p^4, c, d \bigr)$.
\end{proof}

We can use Lemma~\ref{fs} to reduce any given system to one where
 $p \nmid a$, $p^4 \nmid c$, and  $p^2 \nmid d$.
 To this end, repeatedly apply $(pa,pc,d) \sim (pa,pc,p^2d) \sim (a,c,pd)$
 until we get a system with
$p\nmid a$ or $p \nmid c$. Since $(a, c, d) \sim (c, a, d)$ we can
 assume $p\nmid a$. Since
$(a,p^4c,d) \sim (a,c,d)$ and $(a,c,p^2d) \sim (a,c,d)$ 
we can assume $p^4 \nmid c$ and 
$p^2 \nmid d$.

We can go further.

\begin{lemma}
Let $(a,c,d)$ be a system with $p \nmid a$, $p^4 \nmid c$ and 
$p^2 \nmid d$. If $c = p^2 c_0$, and $d = p d_0$ where $p \nmid c_0d_0$,
then the system is not $p$ locally-solvable. In all other cases, $(a,b,c)$ is
equivalent to a system in which at most one coefficient is divisible by~$p$.
\end{lemma}

\begin{proof}
Suppose a system
$(a, p^2c_0, pd_0)$ with
$p\nmid a c_0 d_0$
 has a primitive solution $(u, v, w, z)$ modulo $p^3$.
From $a u^2 + p^2 c_0 w^2 \equiv p d_0 z^2$ modulo~$p^3$, 
it follows that $p \mid u$, which in turn implies $p\mid z$.
From $p\mid u$ and $u w \equiv v^2$ modulo $p^3$, we get 
$p \mid v$, so $p\nmid w$ (the solution is primitive)
and $p^2 \mid u$. Hence $p^3 \mid p^2 c_0 w^2$, a contradiction.

Finally,
$(a,p^3c,pd) \sim (p^4a,p^3c,pd) \sim (p^2a,p c, p d)$ 
$\sim (pa,c,d) \sim (c,pa,d)$, and
$(a,pc,pd) \sim (ap^4,pc,pd) \sim (a p^3, c, d) \sim (c,p^3a,d)$. 
\end{proof}

The following lemmas give effective tests for
$p$-local solvability for the remaining reduced systems.

\begin{lemma}
The systems $(a,c,d)$, $(a,cp,d)$, $(a,c,dp)$, and $(a,cp^3,d)$ 
with $p \nmid acd$  are $p$-locally solvable if and only if they have 
strong solutions modulo $p$ if $p$ is odd, or modulo $16$ if $p=2$.
\end{lemma}

\begin{proof}
This follows from Corollary~\ref{characterizationeasycase} if the system 
is one of the first three. So we consider the system
$(a, c p^3, d)$ where $a, c, d$ are prime to~$p$.

First suppose $p$-locally solvability, and let $(u, v, w, z)$
be a primitive solution modulo~$p^4$. If $(u, v, w, z)$
is not strong, then $p\mid u$ and $p\mid z$.
This forces $p\nmid w$ since $(u, v, w, z)$
is primitive. Since $p\mid v$, we have $p^2\mid u$.
So, $d z^2\equiv 0$ modulo~$p^3$. Thus
$p^2\mid z$. Modulo $p^4$ we now have $c p^3 w^2\equiv 0$, contradicting
the fact that $c$ and $w$ are prime to $p$.  Thus $(u, v, w, z)$ must
be a strong solution modulo $p^4$.

So if the system is $p$-locally solvable it possesses a strong solution modulo $p^4$, and
hence modulo $p$. The converse follows from Theorems~\ref{strongliftingtheorem}
and~\ref{strongliftingtheorem2}.
\end{proof}

\begin{lemma}
The system $(a, c p^2, d)$ with $p \nmid acd$ is $p$-locally solvable 
if and only if 
\begin{enumerate}
\item[(i)] it has a strong solution modulo~$m$, or 
\item[(ii)] the system $( a p^2, c, d )$ has a strong 
           solution modulo $m$. 
\end{enumerate}
Here $m=p$ if  $p$ is odd, but $m=16$ if $p=2$.
\end{lemma}

\begin{proof}
Suppose $p$-local solvability holds, and let $(u, v, w, z)$ be a 
primitive solution modulo $p^6$. If $p\nmid u$, the solution is 
strong. Suppose $p\mid u$. This implies that $v^2 \equiv 0$ and 
$d z^2 \equiv 0$ modulo $p$. So $u, v, z$ are zero modulo~$p$.
Since $(u, v, w, z)$ is primitive, $p\nmid w$. Since
$p\mid v$, we get $p^2 \mid u$. 
Observe that $( u/p^2, v/p, w, z/p)$ is a strong solution modulo $m$ to
the system~$(a p^2, c, d)$.

Conversely, if (i) the system 
$( a, c p^2, d)$ has a strong solution modulo $m$,
then it is $p$-locally solvable by Theorems~\ref{strongliftingtheorem}
and~\ref{strongliftingtheorem2}.
If (ii) the system $( a p^2, c, d )$
has a strong solution modulo $m$,
then the system $( a p^2, c, d )$
is $p$-locally solvable by Theorems~\ref{strongliftingtheorem}
and~\ref{strongliftingtheorem2}.
Since $(a p^2, c, d) \sim (a p^2, c p^4, d) \sim (a, c p^2, d)$ the
system $(a, c p^2, d)$ is $p$-locally solvable as claimed.
 \end{proof}

\begin{exercise}
If $p$ is an odd prime then the test for $p$-local solvability can be
made very explicit. Suppose $p \nmid acd$.
The system $(a, c, d)$ is $p$-locally solvable and the system
$(a, c p^2, dp)$ is not $p$-locally solvable.
Show that the systems $(a, c p, d)$ and $(a, c p^3, d)$ 
are $p$-locally solvable if and only if $ad$ is a square modulo $p$. 
Show that the system $(a, c, d p)$ is $p$-locally solvable if and only if
$-a c^3$ is a fourth power modulo~$p$.
Finally, show that the system $(a, c p^2, d)$ is $p$-locally solvable 
if and only if $a d$ or $c d$ is a square modulo $p$.
\end{exercise}

\subsection*{Relationship with the Quartic}
We now relate the systems studied in this
paper with the nonhomogeneous quartic equation
\begin{equation}\label{singleequation}
a X^4 + b X^2 Y^2 + c Y^4 = d Z^2.
\end{equation}
The main body of this paper treats the $b=0$ case of the system
\begin{equation}\label{system2}
a U^2 + b V^2 + c W^2 = d Z^2, \qquad UW= V^2,
\end{equation}
and the appendices  consider the general case which is important in the study of elliptic curves.
We now show that the solvability of~(\ref{singleequation}) and 
the solvability of~(\ref{system2}) are equivalent
in many situations. In particular, $\Z$-solvability, $\R$-solvability, and $\F_p$-solvability
are covered by the following:

\begin{lemma}
Let $a, b, c, d\in R$ where $R$ is an integral domain. Then
the system~\emph{(\ref{system2})} has a nontrivial $R$-solution
if and only if~\emph{(\ref{singleequation})}
has a nontrivial $R$-solution.
\end{lemma}

\begin{proof}
If $d=0$ then both (\ref{system2}) and (\ref{singleequation}) 
have obvious solutions, so suppose $d\ne 0$.

If $(x_0, y_0, z_0)$ is a nontrivial solution to (\ref{singleequation})
then  $(x^2_0, x_0 y_0, y_0^2, z_0)$
is a nontrivial solution to~(\ref{system2}).

If $(u_0, v_0, w_0, z_0)$ is a nontrivial solution
to~(\ref{system2}), then
both  $(u_0, v_0, z_0 u_0)$ and $(v_0, w_0, z_0 w_0)$
are solutions to (\ref{singleequation}). At least one
is nontrivial since $d\ne 0$.
\end{proof}

The questions of modulo-$p^k$ solvability (for $k>1$) and $p$-local solvability
are covered by the following.
For simplicity we assume $p^2\nmid d$. 

\begin{lemma}\label{redpk}
Let $p$ be a prime and $k>1$.
Suppose that $a, b, c, d\in\Z$ are such that $p^2 \nmid d$.
Then the system  \emph{(\ref{system2})} has a
primitive solution modulo $p^k$
if and only if the equation \emph{(\ref{singleequation})}
has a primitive solution modulo $p^k$.
\end{lemma}

\begin{proof}

If $(x_0, y_0, z_0)$ is a primitive solution to (\ref{singleequation})
modulo~$p^k$ then  $(x^2_0, x_0 y_0, y_0^2, z_0)$
is a primitive solution to~(\ref{system2}) modulo~$p^k$.

If $(u_0, v_0, w_0, z_0)$ is a primitive solution
to~(\ref{system2}) modulo~$p^k$,
then $(u_0, v_0, z_0 u_0)$ and $(v_0, w_0, z_0 w_0)$
are both solution to (\ref{singleequation})
modulo~$p^k$. 
We claim that at least one of $u_0, v_0$ or $w_0$ must be prime to~$p$.
Otherwise, by assumption $z_0$ is prime to~$p$, and
since $d z_0^2\equiv a u_0^2 + b v_0^2 + c w_0^2$ modulo $p^k$,
it follows that $p^2 \mid d$, a contradiction.
With this claim, we see that at least one of the solutions is primitive.
\end{proof}


\section*{Appendix A: The p-adic integers and Hensel's lemma} 

In Section~\ref{s3} we defined $p$-local solvability
in terms of solvability modulo $p^k$. As alluded to in the introduction,
the usual definition of  $p$-local solvability refers to
$p$-adic solutions rather than solutions modulo $p^k$.
In this appendix we sketch an argument that our definition is equivalent to the $p$-adic 
definition. Then we introduce a basic tool, Hensel's lemma, which is
a standard method for finding $\Z_p$-solutions. Both the $p$-adic 
integers and Hensel's lemma will be used in Appendix~B where we discuss 
an important generalization of the
systems considered above. This appendix and the next are designed for readers 
with some familiarity with the $p$-adic numbers.

We begin with a quick review of the $p$-adic numbers.
A $p$-adic integer is a sequence $(a_1, a_2, a_3, \ldots)$
with the following properties for each $k$: (i) $a_k\in \Z/ p^k \Z$, and (ii) the image of $a_{k+1}$
under the natural projection $\Z/ p^{k+1} \Z \rightarrow \Z / p^k \Z$
is equal to~$a_k$.
The set~$\Z_p$ of such sequences forms an integral domain with
addition and multiplication defined componentwise.
Its field of fractions is
denoted by~$\Q_p$.
There is a natural injective ring homomorphism $\Z \rightarrow \Z_p$
defined by sending $a$ to $(a_1, a_2, \ldots )$ where $a_k$ is
the image of $a$ in $\Z / p^k\Z$. Thus we can identify~$\Z$
with a subring of~$\Z_p$ and $\Q$ with a subfield of~$\Q_p$.
The multiplicative structure of $\Z_p$ is simple:
every
nonzero element is uniquely of the form $u p^m$ where $u$
is a unit in~$\Z_p$. The units of $\Z_p$ are the
elements $(a_1, a_2, \ldots)$ such that $a_1$ is a unit in~$\Z / p \Z$.
For an introduction
to the $p$-adic numbers with prerequisites similar to those of the
current paper see~\cite{Gouvea}.

\begin{prop}
For a system of homogeneous equations with coefficients in~$\Z$, the following
are equivalent:
\begin{enumerate}
\item The system has primitive solutions modulo $p^k$ for all $k$.
\item The system has a nontrivial $\Z_p$-solution.
\item The system has a nontrivial $\Q_p$-solution.
\end{enumerate}
\end{prop}

\begin{proof}
The conditions $(2)$ and $(3)$ are equivalent since the system is homogeneous. 
Suppose $(2)$ holds with
 nontrivial $\Z_p$-solution $(x_1, \ldots, x_m)$. Let $p^\lambda$ be the 
largest power of $p$ dividing all the~$x_i$. By dividing each $x_i$ by $p^\lambda$ we can
assume that at least one coordinate is a unit in~$\Z_p$.
Write $x_i= (a_{i 1}, a_{i 2}, \ldots )$. Then, for each $k$, the 
$n$-tuple $(a_{1 k}, \ldots, a_{m k})$ yields
a solution modulo $p^k$. Since there is a coordinate~$x_i$ that is a $\Z_p$-unit,
the corresponding
$a_{ik}$ is a $\Z/ p^k \Z$-unit. So $(a_{1 k}, \ldots, a_{m k})$ is a primitive
solution modulo~$p^k$. We conclude that $(2)\Rightarrow(1)$.

Finally suppose that $(1)$ holds. 
Let $m$ be the number of variables.
To produce a $\Z_p$-solution it is sufficient to produce
 $\mathbf{c_k}=(c_{k 1}, \ldots, c_{k m}) \in\Z^m$
for each~$k$ such that
(i) $\mathbf{c_k}$ is a primitive solution modulo $p^k$
and (ii)~$\mathbf{c_{k+1}}$ is congruent (componentwise) to $\mathbf{c_k}$
modulo $p^k$.
To facilitate the construction, we also consider the condition (iii)
$\mathbf{c_k}$ is infinitely extendable in the following sense:
for all $\lambda\ge k$ there is a primitive solution
modulo~$p^\lambda$
congruent modulo $p^k$ to $\mathbf{c_k}$.

Let $\mathbf{c_1}$ be a primitive solution modulo~$p$ that satisfies (iii).
If no such solution exists then either there is no primitive solution modulo $p$ contradicting $(1)$,
or there are a finite number of modulo-$p$-distinct solutions but where each is not infinitely extendible.
By choosing $\lambda$ large enough, we get a modulus $p^\lambda$ with no 
primitive solutions, also contradicting $(1)$. Thus $\mathbf{c_1}$ exists.
Now suppose $\mathbf{c_1}, \ldots, \mathbf{c_u}$ have been selected for which
 (i) holds for $k \le u$,
(ii) holds for $k<u$, and 
(iii) holds for $k \le u$.
Choose $\mathbf{c_{u+1}}$ to be any primitive solution
modulo $p^{u+1}$  that reduces modulo $p^u$ to $\mathbf{c_u}$ and for which (iii) holds with $k=u+1$.
Such $\mathbf{c_{u+1}}$ exists: otherwise (iii) would fail for $\mathbf{c_u}$.
This construction yields a sequence $\mathbf{c_1}, \mathbf{c_2}, \mathbf{c_3}, \ldots $ 
satisfying (i), (ii), and~(iii) for all $k$. So $(1)\Rightarrow (2)$.
\end{proof}

\emph{Hensel's lemma} refers to a family of results that allows us to
``lift''  modulo $p^k$ solutions 
to $\Z_p$-solutions.
Here is a basic version for polynomials.

\begin{prop}\label{hensel}
Let $f \in \Z_p [T]$ be a polynomial
with derivative~$f'$. If $t \in \Z_p$ is
such that $f(t)\equiv 0$ modulo $p$ but $f'(t) \not \equiv 0$ modulo $p$,
then there is a unique $u \in \Z_p$ such that $f(u)=0$ and such that
$u \equiv t$ modulo $p$.
\end{prop}

\begin{remark}
There are refinements
that deal with the case $f'(t) \equiv 0$ modulo $p$.
\end{remark}


\section*{Appendix B: connections to Elliptic Curves}

In this appendix we assume some familiarity with elliptic curves 
defined over~$\Q$.
Consider the system
\begin{equation}\label{ecappendix}
a U^2 + b V^2 + c W^2 = d Z^2, \qquad UW= V^2
\end{equation}
with $a, b, c, d \in \Z$ such that
 $a, c, d,$ and $b^2 - 4ac$ are nonzero.
Up to this point we have concentrated on the case $b=0$,
which is rich enough to yield 
simple counterexamples to the Hasse principle.
In general, (\ref{ecappendix})  defines a nonsingular projective curve of genus~$1$
given as the intersection of quadric surfaces.\footnote{This curve is
a double cover of the projective planar conic $a U^2 + b UW+ c W^2 = d Z^2$.
By counting ramification points of this cover, and using the Riemann-Hurwitz formula,
one can verify that the genus of the curve is indeed one.
The curve defined by~(\ref{ecappendix}) is an elliptic curve defined over~$\Q$
if and only if~(\ref{ecappendix}) possesses a nontrivial $\Z$-solution.}

Such genus-$1$ curves arise naturally in the $2$-descent procedure
used to find generators and the rank for the group of rational points~$E(\Q)$
of an elliptic curve~$E$.
The system   (\ref{ecappendix}) is adapted
to the case where $E$ is defined over $\Q$ and possesses at least one
$\Q$-rational 2-torsion point, and
the question of the existence of nontrivial  $\Z$-solutions of~(\ref{ecappendix}) 
plays an important role in
the $2$-descent  procedure. 

Another connection between system~(\ref{ecappendix}) and elliptic curves
occurs when (\ref{ecappendix}) is a
counterexample to the Hasse principle.
In that case,  (\ref{ecappendix}) represents an element of
order $2$  of the Tate--Shafarevich group $\sha_E$ of the elliptic
curve~$E$ defined by the equation
$
y^2 = x^3 - 2bd x^2 + (b^2-4ac)d^2 x.
$
For example, 
Lind and Reichardt's  counterexample,
which we studied in the form
$$
U^2 - 17  W^2 = 2 Z^2, \qquad UW= V^2,
$$
represents  an element of order~$2$ in
$\sha_E$ where $E$ is defined by 
$y^2 = x^3 - 2^4 \,17 x$ (which can be transformed into the form
$y^2 = x^3 - 17 x$).
The Tate--Shafarevich group is conjecturally finite, and is
tied to another important conjecture: the conjecture of Birch and
Swinnerton-Dyer.

Much of the study of local solvability considered in the main body of the paper extends to
the system~(\ref{ecappendix}).   For example, we can
generalize Corollary~\ref{plocalcor} as follows:

\begin{thm}\label{T2bx}
The system~\emph{(\ref{ecappendix})} is $p$-locally solvable
for all primes $p \nmid 2acd (b^2 - 4ac)$.
\end{thm}

We end this appendix with the proof of the above theorem.
The idea is to first show that there is a solution modulo~$p$, and then use
Hensel's lemma to derive a $\Z_p$-solution.

One way to prove the existence of solutions modulo $p$
is to use a theorem of
F.~K.~Schmidt (also proved by Ch\^atelet) according to which any
smooth curve of genus $1$ defined over a finite field~$\F_q$
has an $\F_q$-rational point.
Similarly, one can appeal to the Riemann hypothesis for curves over a finite field,
proved by A.~Weil. In this appendix we instead provide an elementary proof.

\begin{lemma}\label{appendix.modplemma}
Let $p$ be an odd prime, and consider
\begin{equation*}
a U^2 + b V^2 + c W^2 = d Z^2, \qquad UW= V^2
\end{equation*}
with $a, b, c, d \in \F_p$, 
and with $a c d (b^2 - 4ac) \ne 0$.
This system  has a nontrivial $\F_p$-solution.
\end{lemma}

\begin{proof}
By multiplying the first equation by $d^{-1}$ we reduce to the case where $d=1$.
We now use the technique of completing the square on
$f(X, Y) = a X^2 + bXY + c Y^2.$
Let $q_1, q_2, q_3 \in \F_p[T]$ be as in
Lemma \ref{qlemma}
applied to $a X^2 + \left( c - \frac{b^2} {4a} \right) Y^2 = Z^2$.
Thus $a q_1^2 + \left( c - \frac{b^2} {4a} \right) q_2^2 = q^2_3$. So,
if $q'_1 =  q_1 - \frac{b}{2a} q_2$ then
$$
\hbox{$f(q'_1, q_2)=
a\left(  q_1 - \frac{b}{2a} q_2 \right)^2 +
b \left(  q_1 - \frac{b}{2a} q_2 \right)q_2 + c q_2^2
=
a q_1^2 +\left(c - \frac{b^2} {4 a}\right) q_2^2
=
q_3^2$.}
$$

Since $q_1$ and $q_2$ are not associates, $q'_1$ is nonzero, and
$q'_1$ and $q_2$ cannot be associates. So, by Lemma~\ref{PA},
there is a $t \in \F_p$ such that
$\big(\frac{q'_1 (t)}{p}\big) \ne - \big(\frac{q_2(t)}{p}\big)$.
So $q'_1(t) q_2(t) = s^2$ for some $s\in\F_p$,
and $q'_1(t)$ and  $q_2(t)$ are not both~$0$.
In particular, $\bigl(q_1'(t), s, q_2(t), q_3(t) \bigr)$ is a 
nontrivial solution.
\end{proof}

The above gives us modulo-$p$ solutions to (\ref{ecappendix}). To
use this to produce $\Z_p$-solutions we need a
special case of Hensel's lemma:

\begin{lemma}\label{specialhenselx}
Let $p$ be a prime not
dividing $2 a c (b^2-4ac)$ where $a, b,c \in\Z$.
If $f(T)=a T^4 + b T^2 + c$ has a root modulo $p$,
then $f$ has a root in $\Z_p$.
\end{lemma}

\begin{proof}
Let $t\in\Z$ be such that $f(t) \equiv 0$ modulo $p$.
Suppose $f'(t) \equiv 0$ modulo~$p$
where $f' = 4 a T^3 + 2 b T$.
In other words, $- 4 a t^3 \equiv 2 b t$ modulo~$p$.
Observe that $t\not\equiv 0$ modulo $p$ since $f(0)=c$ and $p\nmid c$.
Also $p$ is odd.
Thus $ - 2 a t^2 \equiv b$ modulo $p$.
So
$$
0\equiv -  (4 a) a t^4 - (4 a) b t^2 - (4 a) c \equiv  - b^2 + 2b^2 - 4ac
\equiv  b^2 - 4ac \pmod p
$$
contradicting our assumption. Thus $f'(t) \not\equiv 0$ modulo $p$.
The result now follows from Hensel's lemma (Proposition~\ref{hensel}).
\end{proof}

\begin{proof}[Proof of Theorem~\ref{T2bx}.]
Let $(u_0, v_0, w_0, z_0)$ be a 
primitive solution to the system (\ref{ecappendix}) modulo~$p$ 
(Lemma~\ref{appendix.modplemma}).
If $p$ divides both $u_0$ and $w_0$, it must divide $v_0$ and $z_0$ as well,
contradicting the assumption that the solution is primitive. By symmetry between $U$ and $W$
we can assume that $w_0$ is prime to $p$. In particular $w_0$ is a unit in $\Z_p$.

Let $u = u_0 \, w_0^{-1},  \; v = v_0 \, w_0^{-1},$ and $z = z_0\,  w_0^{-1}$ in $\Z_p$. 
Then $(u, v, 1, z)$ also solves (\ref{ecappendix}) modulo~$p$. 
So $u \equiv v^2$, and hence $a v^4 + b v^2 +  c \equiv dz^2$, modulo $p$.

We first consider the case where $z \equiv 0$ modulo $p$.  In this case
$v$ is a root, modulo~$p$, of the polynomial 
$f(T) = a T^4 + b T^2 + c$.
By Lemma~\ref{specialhenselx}, there is a $t \in\Z_p$ such that $f(t) = 0$.
Observe that $\left( t^2, t, 1, 0\right)$ is a $\Z_p$-solution to (\ref{ecappendix}).

Now suppose $z\not \equiv 0$ modulo $p$. Then $z$ is a root, modulo~$p$, of the polynomial 
$f(T) = d\, T^2 - (a v^4 + b v^2 + c)$. Observe that $f'(z) = 2 d z \not\equiv 0$ modulo $p$.
By Hensel's lemma (Proposition~\ref{hensel}) there is a $t\in \Z_p$ so that
$f(t)=0$. In particular $( v^2, v, 1, t)$ is a $\Z_p$-solution to (\ref{ecappendix}).
\end{proof}

\begin{remark}
For more on elliptic curves over~$\Q$, consult~\cite{Cassels},
\cite{ST}, and \cite{Sil}. At a more advanced level see~\cite{Maz}.
\end{remark}


\end{document}